\documentclass{article}

\usepackage{amsfonts}
\usepackage{amsmath}
\usepackage{amssymb}
\usepackage{amsthm}
\usepackage{enumerate}
\usepackage{psfrag}
\usepackage{txfonts}

\newtheorem{definition}{Definition}[section]
\newtheorem{theorem}{Theorem}[section]
\newtheorem{proposition}{Proposition}[section]
\newtheorem{lemma}{Lemma}[section]
\newtheorem{corollary}{Corollary}[section]
\newtheorem{conjecture}{Conjecture}[section]

\title{Minimum Number of Fox Colors for Small Primes}
\author{Pedro Lopes$\sp{1, 2}$ and Jo\~ao Matias$\sp{2}$\\
$\sp{1}$Center for Mathematical Analysis, Geometry and Dynamical Systems\\
$\sp{2}$Department of Mathematics\\
Instituto Superior T\'ecnico\\
Technical University of Lisbon\\
Av. Rovisco Pais\\
1049-001 Lisbon\\
Portugal\\
        \texttt{pelopes@math.ist.utl.pt  joao.matias@ist.utl.pt}\\}

\date{April 2, 2011}
\begin{document}

\maketitle

\begin{abstract}
This article concerns exact results on the minimum number of colors of a Fox coloring over the integers modulo $r$, of a link with non-null determinant.

Specifically, we prove that whenever the least prime divisor of the determinant of such a link and the modulus $r$ is $2, 3, 5$, or $7$, then the minimum number of colors is $2, 3, 4$, or $4$ (respectively) and conversely.

We are thus led to conjecture that for each prime $p$ there exists a unique positive integer, $m_p$, with the following property. For any link $L$ of non-null determinant and any modulus $r$ such that $p$ is the least prime divisor of the determinant of $L$ and the modulus $r$, the minimum number of colors of $L$ modulo $r$ is $m_p$.
\end{abstract}

\bigbreak


\bigbreak


\bigbreak

\section{Introduction}

\noindent

The minimum number of colors in a Fox coloring in a given modulus $r$ is a challenging link invariant since, in principle, it requires analyzing all  diagrams of a given link in order to establish its value. Nonetheless, several results have been obtained, see \cite{Frank, kl, lm, Oshiro, Saito, satoh}, although the problem is far from solved. We now develop the terminology we use in this article.

\bigbreak

Fix a positive integer $r>1$ and consider a diagram $D$ of a link $L$ (\cite{lhKauffman}). A (Fox) $r$-coloring of $D$ is an assignment of integers mod $r$ (also known as ``colors'') to the arcs of $D$ such that, at each crossing the ``coloring condition'' holds i.e., ``twice the color at the over-arc equals the sum of the colors at the under-arcs, mod $r$'' (\cite{Fox}, see Figure \ref{fig:cross}). We let $\mathbf{=_r}$ stand for ``equality modulo $r$''.
\begin{figure}[!ht]
    \psfrag{a}{\huge $a$}
    \psfrag{b}{\huge $b$}
    \psfrag{c}{\huge $c$}
    \psfrag{ast}{\huge $2b=\sb{r}a+c$}
    \centerline{\scalebox{.50}{\includegraphics{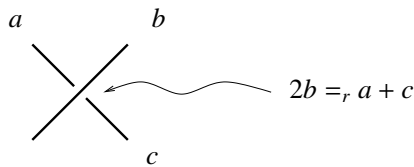}}}
    \caption{Colors and the coloring condition at a crossing.}\label{fig:cross}
\end{figure}
Suppose $D$ is endowed with an $r$-coloring and undergoes a Reidemeister move. Colors can be consistently assigned to the new diagram in order to obtain a new $r$-coloring from the original one, by using the appropriate colors in the new arcs. In this way, there is a bijection between the $r$-colorings, before and after the Reidemeister moves have been performed (\cite{pLopes}). This means that the number of $r$-colorings is a link invariant.

The $r$-colorings of $D$ can also be regarded as the solutions, mod $r$, of the system of homogeneous linear equations we obtain when we regard the arcs of $D$ as algebraic variables and read off each crossing the equation ``twice the over-arc minus the under-arcs equals zero''. This system of homogeneous linear equations and its coefficient matrix are called the ``coloring system of equations'' and the ``coloring matrix'' of $D$, respectively. Upon performance of Reidemeister moves, the new coloring matrix relates to the former by elementary transformations on matrices, as described on page $50$ in \cite{Lickorish}, see also \cite{crowellandfox}. In this way the ideals generated by the minors of this matrix are also link invariants. Since along any row of a coloring matrix the entries are $2, -1, -1$, and $0$'s, then adding all its columns we obtain a column with $0$'s. Thus, the determinant of the coloring matrix is $0$. However, it does follow that the first minors of any two coloring matrices of the link are equal, up to sign, which gives rise to the following link invariant.

\begin{definition}\label{def:det}
 The determinant of a link $L$, denoted $\det L$, is the absolute value of any first minor of any coloring matrix of any diagram of $L$.
\end{definition}

We remark that the determinant of a link becomes the usual determinant of a knot in case the link under study is a knot i.e., a $1$-component link. We thus extend this notion to links. In this article we adhere to links of non-null determinant for it seems that links of null determinant possess a different nature. We intend to study colorings of links with null determinant in another article.

\bigbreak

Given any integer $r>1$ and any link, $L$, there are always $r$-colorings of $L$, the so-called trivial colorings. Each of these $r$-colorings is obtained by assigning, in any diagram of $D$, the same color to each of its arcs. There are exactly $r$ trivial $r$-colorings of $L$. The consistent assignment of colors we mentioned before, upon performance of Reidemeister moves, takes trivial $r$-colorings to trivial $r$-colorings. Thus it takes non-trivial $r$-colorings (i.e., $r$-colorings which use more than one color) to non-trivial $r$-colorings.

If  a link $L$ admits non-trivial $r$-colorings, it would be interesting to know what is the least number of colors it takes to set up a non-trivial $r$-coloring of $L$ over all diagrams of $L$.

\begin{definition} Let $r$ be an integer greater than $1$. Let $L$ be a link admitting non-trivial $r$-colorings, let $D$ be a diagram of $L$ and let $n\sb{r, D}$ be the minimum number of colors mod $r$ it takes to construct a non-trivial $r$-coloring of $D$. Set
\[
mincol\sb{r}L := \min \{ \, n\sb{r, D} \, \mid \, D \text{ is a diagram of } L  \, \}
\]
We call $mincol\sb{r}L$ {\rm the minimum number of colors for the link $L$}.
\end{definition}

$mincol\sb{r}L$ is tautologically a link invariant. This invariant was set forth in \cite{Frank}. In this article it was also conjectured that given a prime $p$ and an alternating
knot $L$ of prime determinant $p$, each non-trivial $p$-coloring on any minimal diagram of $L$ assigns distinct colors to distinct arcs. This came to be known as the Kauffman-Harary conjecture and is now solved in \cite{msolis}.

\bigbreak

Now that the Kauffman-Harary conjecture has been proven to be true, the current article looks into the other possibilities for colorings of links without the constraint of using a prime modulus for the colorings which coincides with the determinant of the link under study. Moreover, this article is not a survey. It is a self-contained research article which acts as a survey on certain technical details.

\bigbreak

\bigbreak

In \cite{Frank}, an interesting example of Teneva showed that we could further reduce the number of colors of a non-trivial $p$-coloring of an alternating knot of prime determinant $p$ if we used non-minimal diagrams, for the particular case $p=5$ and the knot $5\sb{1}$, also known as the torus knot of type $(2, 5)$. This issue was investigated and developed, leading to the article \cite{kl} where exact values and estimates for the minimum number of colors of the torus knots of type $(2, n)$ were presented. These exact values and these estimates were expressed in terms of the least prime divisor of the modulus $r$ (with respect to which the colorings were being considered) and of the crossing number of the $T(2, n)$ under study (or so it was believed at that time). We note that, for the $T(2, n)$'s, the determinant and the crossing number are equal. And the current article illustrates that, besides the modulus $r$, the second relevant feature seems to be the determinant of the link. In fact, we express our exact results as a function of the least prime divisor of the modulus $r$ and of the determinant of the link under study. Here we introduce the following notation which simplifies the statement of our main result.

\bigbreak

\begin{definition} Let $a$ and $b$ be positive integers.

We let $(a, b)$ stand for their greatest common divisor. Occasionally we will let one of the arguments be $0$ in which case $(a, b)$ will stand for the other argument.

We let $\langle a, b\rangle$ stand for $1$ if $(a, b)=1$, else we let $\langle a, b\rangle$ stand for their least common prime divisor. Occasionally we will let one of the arguments be $0$ in which case $\langle a, b\rangle$ will stand for the least prime divisor of the other argument.
\end{definition}

We can now state our main result concisely as follows.

\bigbreak

\begin{theorem}\label{thm:exactmin}Let $r$ be an integer greater than $1$.

Let $L$ be a link with $\det L \neq 0$. Then

\begin{enumerate}
\item  \quad $\langle r, \det L \rangle = 2$  \quad if and only if  \quad $mincol\sb{r}L = 2$
\item   \quad $\langle r, \det L \rangle = 3$  \quad if and only if  \quad $mincol\sb{r}L = 3$
\item   \quad $\langle r, \det L \rangle \in \{ 5, 7 \}$  \quad if and only if  \quad $mincol\sb{r}L = 4$
\item    \quad $\langle r, \det L \rangle > 7$, \quad if and only if  \quad $mincol\sb{r}L \geq 5$
\end{enumerate}
\end{theorem}

\bigbreak

This theorem is important because it solves the problem of the minimum number of colors for each link $L$ with non-null determinant and for each modulus $r$ such that $\langle r, \det L \rangle$ is either $2, 3, 5$ or $7$. Moreover it does so in a uniform way in the sense that the feature of the link that matters for the theorem is whether its determinant is relatively prime to $r$ or not. If they are relatively prime, there are no non-trivial $r$-colorings. If they are not relatively prime, then the relevant parameter for Theorem \ref{thm:exactmin} is $\langle r, \det L \rangle$. This leads us to conjecture the same sort of behavior for the other primes that may occur as $\langle r, \det L \rangle$.

\bigbreak

\begin{conjecture}\label{conj:conj}
Let $r$ be an integer greater than $1$, let $L, L'$ be links with non-null determinants.
\begin{itemize}
\item If $\langle r, \det L \rangle = \langle r, \det L' \rangle $, then $mincol\sb{r} L = mincol\sb{r} L'$

We then set, for each prime $p$
\[
m\sb{p}:= mincol\sb{r}L \qquad \text{ for any $L$ such that $p=\langle r, \det L \rangle $ and $\det L \neq 0$}
\]
In particular,
\[
m\sb{2}=2, \qquad \qquad m\sb{3}=3, \qquad \qquad m\sb{5}=m\sb{7}=4
\]

\item $m\sb{p}$ is a non-decreasing function of $($prime$)$ $p$
\end{itemize}
\end{conjecture}

\bigbreak

A ``split link'' is a link which can be made to be contained in two disjoint ball neighborhoods, each neighborhood containing part of the link. A ``non-split'' link is a link which is not a ``split link''. We remark that the minimum number of colors for a non-trivial coloring of a split link is trivially $2$. We further remark that split links have null determinant. In particular, by adhering to links with non-null determinant in this article, we are disregarding trivial problems.

\bigbreak

The following observations on Theorem \ref{thm:exactmin} are in order. It seems more appropriate to make them here than to over-burden the statement of the Theorem. The ``only if'' part of statement $1$ of the Theorem is true without the condition ``$\det L \neq 0$''; the ``if'' part of statement $1$ is true with ``$\det L \neq 0$'' replaced by the weaker condition ``non-split link $L$''. For the ``only if'' part of statement $2$, ``$\det L \neq 0$'' can be replaced by ``non-split link $L$''; for the ``if'' part, ``$\det L \neq 0$'' can be removed. For the ``only if'' part of statement $3$,  ``$\det L \neq 0$'' can be replaced by  ``non-split link $L$''. The veracity of these facts follows from Propositions \ref{prop:2}, \ref{prop:3}, and \ref{prop:5,7}.

\bigbreak

One of the underpinnings of Theorem \ref{thm:exactmin} is Lemma \ref{lem:mainlemma}. We state it here because we believe it is of independent interest.

\begin{lemma}\label{lem:mainlemma} Let  $L$ be a non-split link admitting non-trivial $r$-colorings. There exists a prime $p\mid (r, \det L)$ such that
\[
mincol\sb{r}L = mincol\sb{p}L
\]
\end{lemma}

\bigbreak

\subsection{Organization and Acknowledgements}\label{subsect:ackn}

\noindent

This article is organized as follows. Section \ref{sect:class} is devoted to known results. In Subsection \ref{subsect:det} we go over the notion of determinant of a link and some of its consequences in order to bring out its relationship to the minimum number of colors. In particular, for any integers $p\mid r$, and link $L$, it is shown that a non-trivial $p$-coloring of $L$ gives rise to a non-trivial $r$-coloring of $L$. It follows that $mincol\sb{r}L \leq mincol\sb{p}L$. Although we believe this to be known material, we were not able to find a complete reference for it, so we prove this result here. Moreover, this contrasts with a sort of converse proven in Section \ref{sect:proof}, namely that for every non-trivial $r$-coloring, there exists a non-trivial $p$-coloring for some $p\mid r$ with the same number of colors. It follows then that $mincol\sb{r}L = mincol\sb{p}L$. This technical result (Lemma \ref{lem:mainlemma}) is crucial in our findings. Furthermore, Lemma \ref{lem:mainlemma} is a new result, to the best of our knowledge. In Subsection \ref{subsect:min23} we recall a few results relating the minimum number of colors with the divisibility of the modulus, the proofs being presented here for completeness. In Subsection \ref{subsect:saitoproof} we prove a variation on a result of Saito's in \cite{Saito}. In Section \ref{sect:proof} we prove Lemma \ref{lem:mainlemma} and Theorem \ref{thm:exactmin}.
%In Section \ref{sect:finrem} we state a conjecture derived from the features of Theorem \ref{thm:exactmin}.

\bigbreak

P.L. and J.M. acknowledge support by the Funda\c{c}\~{a}o para a Ci\^{e}ncia e a Tecnologia (FCT /
Portugal). They  also gratefully acknowledge Louis Kauffman for his helpful remarks.

\bigbreak

\section{Background Results}\label{sect:class}

\noindent

This Section is a review of known material whose proofs are either scattered through several references or difficult to find. We present this material because it is essential to the proof of Theorem \ref{thm:exactmin}.

\bigbreak

\subsection{The Determinant}\label{subsect:det}

\noindent

In this Subsection we recall a few facts involving the determinant  of a link. These facts belong to the folklore of knot theory, the best reference being perhaps \cite{livingston}. Our purpose is to pave the way to showing the dependence of the minimum number of colors on the determinant of the link.

\bigbreak

\begin{proposition}\label{prop:detcol} Let $r$ be an integer greater than $1$. Let $L$ be a link. There are non-trivial $r$-colorings of $L$ if and only if $(r, \det L) \neq 1$.
\end{proposition}
 \begin{proof} The coloring matrix of any link $L$ is an integer matrix, with one $2$, two $-1$'s, and $0$'s along each row; in particular, its determinant is $0$,  as we saw above. The coloring matrix can thus  be diagonalized by way of a finite number of elementary transformations, and one of the entries of the diagonal in the diagonalized matrix is zero. The product of the remaining elements in this diagonal is $\det L$, the determinant of the link $L$. Suppose we are working over a modulus $r$ which is relatively prime to $\det L$. There is then just one zero on the diagonal and the variable corresponding to the zero along the diagonal can take on any value $0, 1, 2, \dots , r-1 $ (mod $r$), whereas the other variables have to take on the value $0$ (mod $r$). There are then $r$ distinct solutions of the coloring system of equations and each one represents a trivial $r$-coloring.

Assume now $r$ and $\det L$ are not relatively prime and let $p$ be a common prime factor of $r$ and $\det L$. Then $p\mid \det L$ and so, mod $p$, there are at least two $0$'s mod $p$ along the diagonal. Hence at least two variables can take on any value from $\{ 0, 1, 2, \dots , p-1 \}$ mod $p$. That is, there are non-trivial $p$-colorings. Let $D$ be a diagram of $L$ endowed with such a non-trivial $p$-coloring, with $(c\sb{i})$ the sequence of colors, mod $p$, the arcs of this diagram take on. Choose the $c\sb{i}$'s such that $0\leq c\sb{i} < p$, for each $i$. Consider the triples $(i\sb{1}, i\sb{2}, i\sb{3})$ such that:
\[
2c\sb{i\sb{1}}-c\sb{i\sb{2}}-c\sb{i\sb{3}}=p\, l\sb{i\sb{1}, i\sb{2}, i\sb{3}} \qquad \text{ at some crossing of $D$ with } l\sb{i\sb{1}, i\sb{2}, i\sb{3}} \text{ an integer}
\]
Multiplying each of these equations by $\frac{r}{p}$ we obtain:
\[
2\frac{r}{p}c\sb{i\sb{1}}-\frac{r}{p}c\sb{i\sb{2}}-\frac{r}{p}c\sb{i\sb{3}}=r\, l\sb{i\sb{1}, i\sb{2}, i\sb{3}} \qquad \text{ at some crossing of $D$ with } l\sb{i\sb{1}, i\sb{2}, i\sb{3}} \text{ an integer}
\]
i.e. we obtain an $r$-coloring of $D$ with colors $(\frac{r}{p}c\sb{i})$. Since the $c\sb{i}$'s form a non-trivial $p$-coloring, there are indices $i\sb{0}\neq i\sb{0}'$ such that $0\leq c\sb{i\sb{0}}<c\sb{i\sb{0}'}<p$.  Then, multiplying by $\frac{r}{p}$ we obtain $0\leq \frac{r}{p}c\sb{i\sb{0}} < \frac{r}{p}c\sb{i\sb{0}'}< r$ and  so $(\frac{r}{p}c\sb{i})$ constitutes a non-trivial $r$-coloring of $D$. This concludes the proof.
\end{proof}

The construction of an $r$-coloring out of a $p$-coloring that we did in the proof of Proposition \ref{prop:detcol} proves the following statement.

\begin{lemma}\label{lem:pcolrcol} We keep the notation from Proposition \ref{prop:detcol}. For any positive integer $r_0>1$ such that  $r_0\mid (r, \det L )$, a non-trivial $r_0$-coloring of $L$ gives rise to a non-trivial $r$-coloring of $L$ with the same number of colors, by multiplying each color of the non-trivial $r_0$-coloring by $\frac{r}{r_0}$, and reading them and the equations at the crossings mod $r$.

In particular,  $mincol\sb{r}L \leq mincol\sb{r_0}L$.

\end{lemma}

\bigbreak

\subsection{Minimum number of colors: $2$ and $3$}\label{subsect:min23}

\noindent

We now prove a result which can be found in \cite{kl}. We state and prove it here for completeness.

\begin{proposition}\label{prop:2and3}
Let $L$ be a non-split link and $r$ be an integer greater than $1$. Assume further $L$ admits non-trivial $r$-colorings.
\begin{enumerate}
\item If $mincol\sb{r}L = 2$ then $2|r$

\item If $mincol\sb{r}L = 3$ then $3|r$
\end{enumerate}
\end{proposition}
\begin{proof}
\begin{enumerate}
\item If $mincol\sb{r}L = 2$ then there is a diagram, $D$, of $L$ endowed with a non-trivial $r$-coloring which uses two distinct colors, call them $a$ and $b$, with $0\leq a < b <r$ (without loss of generality). At a crossing of this diagram, these two colors have to meet. The possibilities are:
    \begin{enumerate}
\item One of the under-arcs bears $a$, the other under-arc and the over-arc bear $b$. The equation read at the crossing yields $2b=\sb{r}a+b$ which amounts to $a=\sb{r}b$, which is absurd.
\item The under-arcs bear $a$ and the over-arc bears $b$. The equation is then $2b=\sb{r}a+a$ which amounts to $2(b-a)=\sb{r}0$. Since $0 < b-a < r$, then $0<2(b-a)< 2r$ so $2(b-a)=r$ i.e., $2|r$. In particular, $b=a+\frac{r}{2}$ (this will be useful when analyzing the next case).
\end{enumerate}

\item If $mincol\sb{r}L = 3$ then there is a diagram, $D$, of $L$ endowed with a non-trivial $r$-coloring which uses three distinct colors, call them $a, b$ and $c$, all of them chosen between $0$ and $r-1$. Suppose that, at crossings where more than one color meet, only two colors meet. Then, without loss of generality, $2(b-a)=\sb{r}0$ and $2(c-a)=\sb{r}0$ which imply $b= a+\frac{r}{2}= c $, conflicting with the assumption of three distinct colors. Therefore at some crossing the three distinct colors meet yielding $c=\sb{r}2b-a$.
\begin{figure}[!ht]
    \psfrag{0}{\huge $a$}
    \psfrag{1}{\huge $b$}
    \psfrag{2c}{\huge $c$}
    \psfrag{c'}{\huge $x$}
    \psfrag{2c'-1}{\huge $y$}
    \centerline{\scalebox{.50}{\includegraphics{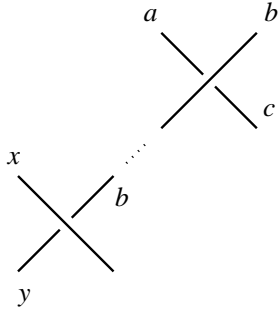}}}
    \caption{The three distinct colors $a, b, c$ meeting at the upper crossing}\label{fig:2x2}
\end{figure}\\
Since the link is non-split, then the over-arc bearing $b$ has to come to a crossing with respect to which it is an under-arc and the colors meeting are distinct in two's. Perhaps this over-arc bearing $b$ first encountered crossings where there was only the color $b$ meeting, but after a fashion an over-arc bearing the color $b$ becomes an under-arc at a crossing where the three colors are distinct (see bottom left portion of Figure \ref{fig:2x2}). Then either $(x, y)=(a, c)$ in which case $c=\sb{r}2a-b$ which in conjunction with $c=\sb{r}2b-a$ yields $3(b-a)=\sb{r}0$ which in turn yields $3|r$. Or $(x, y)=(c, a)$ leading to the same result. This concludes the proof.
\end{enumerate}
\end{proof}

\bigbreak

\subsection{A variation on a result of Saito's}\label{subsect:saitoproof}

\noindent

In this Subsection we prove the following variation of  a result  of Saito's (\cite{Saito}). The proof we give is simpler under the hypothesis of $\det L\neq 0$ and we include it here for completeness.

\bigbreak

\begin{theorem}\label{thm:extsaito} Let $L$ be a link with $\det L \neq 0$. If there is a prime $p>7$ such that $L$ admits non-trivial $p$-colorings, then $mincol\sb{q}L \geq 5$ for any prime $q>7$ for which $L$ admits non-trivial $q$-colorings.
\end{theorem}
\begin{proof} This proof is adapted from  Saito's proof of Theorem 1.1 (1) in \cite{Saito}. It differs from it on the handling of two cases and on the final use of $\det L \neq 0$.

Let $L$ and $p$ be as in the statement. Since the minimum number of colors cannot be $2$ nor $3$ (for $2\nmid p$, and $3\nmid p$,  see Proposition \ref{prop:2and3}), assume there is a diagram $D$ of $L$ endowed with a non-trivial $p$-coloring and using four distinct colors. Arguing as in the proof of Proposition \ref{prop:2and3}, there is a crossing of this diagram where three distinct colors meet. We recall that colorings are solutions of linear systems of equations. Therefore, linear combinations of solutions are solutions. Thus, subtracting the standing non-trivial $p$-coloring by the appropriate trivial $p$-coloring, we obtain a new non-trivial $p$-coloring using four colors and such that, at the indicated crossing, one of the under-arcs bears color $0$. Assume now the over-arc at this crossing bears color $1$ (otherwise, divide mod $p$ each color by this one). Then the other under-arc at this crossing bears color $2$. There is only one color left to be determined. Arguing as in the proof of Proposition \ref{prop:2and3}, the arc bearing $1$ or another arc stemming from it and also bearing color $1$, will become an under-arc at a crossing where all the three colors meeting are distinct. Now these three colors cannot be $0, 1, 2$ for otherwise $3\mid p$ which contradicts the hypothesis. Hence, at this new crossing the fourth color has to be involved. Here are the possibilities.

The over-arc cannot bear the color $1$. Suppose the over-arc bears $0$, in which case the other under-arc bears $-1$. Upon adding the trivial coloring formed by the color $1$, we can assume the non-trivial coloring here to be formed by $0, 1, 2, 3$.

Now assume the over-arc bears $2$. Then the other under-arc bears color $3$ and the  four colors at issue are $0, 1, 2, 3$.

Finally, assume the over-arc bears the yet unknown fourth color, call it $x$. The possibilities for the remaining under-arc are $0$ or $2$. Suppose this under-arcs bears $0$. Then, $2x=\sb{p}1$ which upon multiplying each color by $2$ (i.e., adding this solution to itself) amounts to the following four colors: $0, 1, 2, 4$. Suppose the under-arc bears color $2$. Then, $2x=\sb{p}1+2=3$ and upon multiplication by $2$ we can assume the four colors to be $0, 2, 3, 4$.

Therefore, we are left with the following three possibilities for the four colors:
\[
0, 1, 2, 3 \qquad \qquad 0, 1, 2, 4  \qquad \qquad 0, 2, 3, 4
\]

Let us consider the first set of colors, $\{  0, 1, 2, 3  \}$. What are the possibilities for the coloring condition at a generic crossing of the diagram with these four colors besides the monochromatic possibility (over-arc and under-arcs receive the same color)? $3$ cannot be the color of an over-arc at a polichromatic crossing. $2$ can be the color of an over-arc at a polichromatic crossing if the under-arcs receive colors $3$ and $1$. Likewise $1$ can be the color of an over-arc at a polichromatic crossing if the under-arcs receive colors $0$ and $2$. $0$ cannot be the color of an over-arc at a polichromatic crossing. Furthermore, we realize that the possible coloring conditions just listed hold over the integers and hence hold modulo any prime $q>7$. Had we started with the sets of colors $\{  0, 1, 2, 4  \}$ and $\{  0,  2, 3, 4  \}$ we would have arrived at the same conclusions about the possible coloring conditions. They would have hold over the integers. We can then conclude that there is a non-trivial $q$-coloring of the link $L$ for infinitely many primes $q$. But since $\det L \neq 0$, then only a finite number of primes divide $\det L$ i.e., only a finite number of primes $q$ may give rise to non-trivial $q$-colorings of $L$. This contradiction concludes the proof.
\end{proof}

\bigbreak

\section{Proof of Theorem \ref{thm:exactmin}}\label{sect:proof}

\noindent

Having established the facts of Section \ref{sect:class}, the proof of Theorem \ref{thm:exactmin} is now a non-trivial application of combinatorics and number theory to them and to Lemma \ref{lem:mainlemma}. We thus start this Section by proving Lemma \ref{lem:mainlemma}.

\bigbreak

\begin{proof}[Proof of Lemma \ref{lem:mainlemma}] If $r$ is prime then $p=r$.

If $r$ is composite, we prove the lemma  in $3$ steps, as follows.

\begin{enumerate}
\item  Assume $r$ is composite and $mincol\sb{r}L = n$. Without loss of generality, suppose this minimum number of colors is realized over a diagram $D$ of $L$ with colors $0 = c\sb{0} < c\sb{1}<\dots <c\sb{n-1}<r$ (mod $r$).

\item Let $g$ be the greatest common divisor of $r, 0, c\sb{1}, \dots , c\sb{n-1}$. Note that $g<r$ for otherwise the coloring would be trivial. If $g=1$ go to $\rm 3.$. Otherwise, consider the set of triples $(i\sb{1}, i\sb{2}, i\sb{3})$ such that
\[
2c\sb{i\sb{1}}- c\sb{i\sb{2}} - c\sb{i\sb{3}} = l\sb{i\sb{1}, i\sb{2}, i\sb{3}}r
\]
holds at a crossing of $D$, where $l\sb{i\sb{1}, i\sb{2}, i\sb{3}}$ is an integer. Dividing throughout by $g$, we obtain:
\[
2\frac{c\sb{i\sb{1}}}{g}- \frac{c\sb{i\sb{2}}}{g} - \frac{c\sb{i\sb{3}}}{g} = l\sb{i\sb{1}, i\sb{2}, i\sb{3}}\frac{r}{g}
\]
Moreover, with integer $l\sb{i, j}$
\[
c\sb{i}-c\sb{j}=l\sb{i, j}r  \qquad \text{ is equivalent to saying } \qquad \frac{c\sb{i}}{g}-\frac{c\sb{j}}{g}=l\sb{i, j}\frac{r}{g}
\]
which means there is an induced non-trivial $\frac{r}{g}$-coloring of $D$ mod $\frac{r}{g}$ using colors $0< \frac{c\sb{1}}{g}< \dots < \frac{c\sb{n-1}}{g} \, \big( <\frac{r}{g}  \big)$ and
\[
n = mincol\sb{r}L \leq mincol\sb{\frac{r}{g}}L \leq n
\]
using Lemma \ref{lem:pcolrcol} for the middle inequality. Thus, $n = mincol\sb{r}L = mincol\sb{\frac{r}{g}}L$.

\item For any prime $p\mid \frac{r}{g}$, say $\frac{r}{g}=r\sb{0}p$, we have
\[
2\frac{c\sb{i\sb{1}}}{g}- \frac{c\sb{i\sb{2}}}{g} - \frac{c\sb{i\sb{3}}}{g} = l\sb{i\sb{1}, i\sb{2}, i\sb{3}}\frac{r}{g} = l\sb{i\sb{1}, i\sb{2}, i\sb{3}}r\sb{0}p
\]
so we have an induced $p$-coloring of $D$, with colors $\frac{c\sb{i}}{g}$ read mod $p$. Suppose they are all equal mod $p$, so, necessarily all equal to $0$ mod $p$:
\[
0=\frac{c\sb{i}}{g}=l\sb{i}p \qquad \text{ for } i=1, \dots , n-1 \quad \text{ and integer } l\sb{i}
\]
Then, $p$ divides $\frac{r}{g}$ and each $\frac{c\sb{i}}{g}, i=1, \dots , n-1$. But this contradicts the standing assumption that $g$ is the greatest common divisor of $r, 0, c\sb{1}, \dots , c\sb{n-1}$. Hence, at least two of the colors $\frac{c\sb{i}}{g}, i=0, \dots , n-1$ are distinct, mod $p$ ($\frac{c\sb{0}}{g}=0$) and so all of them constitute a non-trivial $p$-coloring of $D$. Let us denote these colors $c\sb{i}\sp{p}$, with
\[
0 = c\sb{0}\sp{p} \leq  c\sb{1}\sp{p}\leq c\sb{2}\sp{p}\leq  \dots \leq c\sb{n-1}\sp{p}
\]
so $mincol\sb{p}L\leq n$.
Using Lemma \ref{lem:pcolrcol} and arguing as in the end of $\rm 2.$, we obtain $mincol\sb{p}L=mincol\sb{r}L$, concluding the proof.
\end{enumerate}
\end{proof}

\bigbreak

\begin{corollary}\label{cor:2}
\[
mincol\sb{r}L = \min \{ \, mincol\sb{p}L \, | \, p \, \text{ is prime and }  p\mid (r, \det L )   \}
\]
\end{corollary}
\begin{proof} This follows from applying Lemmas \ref{lem:mainlemma} and \ref{lem:pcolrcol}.
\end{proof}

\bigbreak

\begin{proposition}\label{prop:2}Let $L$ be a link. If
$\langle r, \det L \rangle = 2$  then $mincol\sb{r}L = 2$.

Let $L$ be a non-split link. If $mincol\sb{r}L = 2$ then $\langle r, \det L \rangle = 2$.
\end{proposition}
\begin{proof}If $\langle r, \det L \rangle = 2$, then $2\mid r$ and $2\mid \det L$. By Lemma \ref{lem:pcolrcol} there is a non-trivial $r$-coloring of $L$ with $2$ colors. Since a non-trivial coloring has to use at least $2$ distinct colors then $mincol\sb{r}L=2$.

Now assume $mincol\sb{r}L = 2$ with $L$ a non-split link. Then, there is a non-trivial $r$-coloring of $L$ with $2$ colors. By Lemma \ref{lem:mainlemma} there is a prime $p\mid (r, \det L )$ such that $mincol\sb{p}L=mincol\sb{r}L=2$. If $L$ admits non-trivial $2$-colorings, then $mincol\sb{2}L=2$. We are going to show that $2$ is the sole prime with this property. Suppose $p$ is prime and there is a non-trivial $p$-coloring over diagram $D$ of $L$ which uses two distinct colors. Then by $1$. of Proposition \ref{prop:2and3}, $2\mid p$ and since $p$ is prime then $p=2$. By Proposition \ref{prop:detcol} $(p, \det L )\neq 1$ and since $p=2$ is prime, $\langle p, \det L \rangle = 2$, which completes the proof.
\end{proof}

\bigbreak

\begin{proposition}\label{prop:3} Let $L$ be a non-split link. If
$\langle r, \det L \rangle = 3$,  then $mincol\sb{r}L = 3$.

If $L$ is a link with $mincol\sb{r}L = 3$, then $\langle r, \det L \rangle = 3$.
\end{proposition}
\begin{proof}Let $L$ be a non-split link with $\langle r, \det L \rangle = 3$. Then $3\mid r$ and $3\mid \det L$. By Lemma \ref{lem:pcolrcol} there is a non-trivial $r$-coloring of $L$ with $3$ colors. It cannot use less than $3$ colors for otherwise by Proposition \ref{prop:2}, $\langle r, \det L \rangle = 2$, which conflicts with the hypothesis. Then $mincol\sb{r}L=3$.

Now assume $mincol\sb{r}L = 3$. Then, there is a non-trivial $r$-coloring of $L$ with $3$ colors and $L$ is necessarily a non-split link. By Lemma \ref{lem:mainlemma} there is a prime $p\mid (r, \det L )$ such that $mincol\sb{p}L=mincol\sb{r}L=3$. By $2$. of Proposition \ref{prop:2and3}, if $mincol\sb{p}L=3$, then $3\mid p$ and since $p$ is prime then $p=3$. By Proposition \ref{prop:detcol} $(p, \det L)\neq 1$ and since $p=3$ is prime then $\langle p, \det L\rangle =3$ which concludes the proof.
\end{proof}

\begin{proposition}\label{prop:5,7} Let $L$ be a non-split link.  If $\langle r, \det L \rangle \in \{ 5, 7 \}$ then $mincol\sb{r}L = 4$.

Let $L$ be a link with $\det L\neq 0$. If $mincol\sb{r}L = 4$ then $\langle r, \det L \rangle \in \{ 5, 7 \}$
\end{proposition}
\begin{proof}Let $L$ be a non-split link. If $\langle r, \det L \rangle = 5$, then $5\mid r$ and $5\mid \det L$. In particular, there is a non-trivial $5$-coloring of $L$. By the result of Satoh (\cite{satoh}), there is a non-trivial $5$-coloring of $L$ using only $4$ colors, and $4=mincol_5 L$. By Lemma \ref{lem:pcolrcol} there is a non-trivial $r$-coloring of $L$ with $4$ colors. It cannot use less than $4$ colors for otherwise by Propositions \ref{prop:2} and \ref{prop:3}, $\langle r, \det L \rangle \in \{ 2, 3 \}$, which conflicts with the hypothesis. Then $mincol\sb{r}L=4$. Likewise, if $\langle r, \det L \rangle = 7$, we invoke the result of Oshiro (\cite{Oshiro}) to obtain a non-trivial $7$-coloring of $L$ using only $4$ colors (with $4=mincol_7 L$) to arrive at $mincol\sb{r}L=4$.

Now assume $L$ is a link with $\det L\neq 0$ and $mincol\sb{r}L = 4$. Then, there is a non-trivial $r$-coloring of $L$ with $4$ colors. By Lemma \ref{lem:mainlemma} there is a prime $p\mid (r, \det L )$ such that $mincol\sb{p}L=mincol\sb{r}L=4$. Then, Theorem \ref{thm:extsaito} together with Propositions \ref{prop:2} and \ref{prop:3}, imply that $p\in \{ 5, 7 \}$, which concludes the proof.
\end{proof}

\bigbreak

\begin{proof}[Proof of Theorem \ref{thm:exactmin}] Statements {\rm 1., 2., 3.} in the Theorem are weaker versions of the statements in Propositions \ref{prop:2}, \ref{prop:3}, and \ref{prop:5,7}, respectively, which we have already proved. Then
\[
\langle r, \det L \rangle \notin \{ 2, 3, 5, 7 \}, \quad \text{ if, and only if, } \quad mincol\sb{r}L \geq 5
\]
otherwise conflicting at least with one of these Propositions. This concludes the proof of Theorem \ref{thm:exactmin}.
\end{proof}

\bigbreak

\bigbreak

\end{document}